%% file: GCF_without_singularities_041916.tex
\documentclass[12pt]{amsart}
\usepackage[centering]{geometry}

\usepackage{amsmath,amsthm,amssymb}
\usepackage{times}
\usepackage{enumerate}
\usepackage{graphicx, color}

\theoremstyle{plain}
\newtheorem*{theorem*}{Theorem}
\newtheorem{theorem}{Theorem}[section]

\newtheorem{lemma}[theorem]{Lemma}
\newtheorem{proposition}[theorem]{Proposition}

\newcommand{\be}{\begin{equation}}
\newcommand{\ee}{\end{equation}}

\newcommand{\lt}{\left}
\newcommand{\rt}{\right}
\newcommand{\goto}{\rightarrow}
\newcommand{\R}{\mathbb{R}}
\newcommand{\e}{\epsilon}
\newcommand{\kp}{\kappa}
\newcommand{\vp}{\varphi}
\newcommand{\tu}{\tilde{u}}
\newcommand{\tC}{\tilde{C}}

\theoremstyle{definition}
\newtheorem{defin}[theorem]{Definition}
\newtheorem{remark}[theorem]{Remark}

\numberwithin{equation}{section}
\numberwithin{equation}{section}

\begin{document}
\setlength{\baselineskip}{1.2\baselineskip}

\title[General curvature flow without singularities]
{General curvature flow without singularities}

\author{Ling Xiao}
\address{Department of Mathematics, Rutgers University,
Piscataway, NJ 08854}
\email{lx70@math.rutgers.edu}

\begin{abstract}

In \cite{SS}, S{\'a}ez and Schn{\"u}rer studied the graphical mean curvature flow of complete hypersurfaces defined on subsets of Euclidean space.
They obtained long time existence. Moreover, they provided a new interpretation of weak mean curvature flow. In this paper, we generalize their results
to a general curvature setting. Our key ingredient is the existence result of general curvature flow with boundary conditions, which is proved in Section \ref{ea}.

\end{abstract}

\maketitle

\section{Introduction}
\label{int}
\setcounter{equation}{0}

In \cite{SS}, S{\'a}ez and Schn{\"u}rer studied the graphical mean curvature flow of complete hypersurfaces defined on subsets of Euclidean space.
They obtained long time existence. Moreover, they provided a new interpretation of weak mean curvature flow (see Section 1 in \cite{SS}).
In this paper, we generalize their results to a general curvature setting. More specifically, let $\Sigma_0$ be a complete Weingarten
hypersurface in $\R^{n+1}$ satisfies $f(\kappa[\Sigma_0])>0,$ and is given by an embedding $X(0):M^n\goto\R^{n+1},$ we consider the evolution of such an embedding to produce a family of embeddings $X(t)$ satisfies
\be\label{in.1}
\left\{\begin{aligned}
\dot{X}&=f(\kp[\Sigma])\nu\\
X(0)&=\Sigma_0,
\end{aligned}\right.
\ee
where $\kp[\Sigma(t)]=(\kp_1, \cdots, \kp_n)$ denotes the principal curvature of $\Sigma(t),$ and $\nu$ is the unit normal vector.

We assume the function $f$ satisfies the following fundamental structure conditions:
\be\label{in.2}
f_i(\lambda)\equiv\frac{\partial f(\lambda)}{\partial\lambda_i}>0,\,\,\mbox{in $\Gamma$}, 1\leq i\leq n,
\ee
\be\label{in.3}
\mbox{$f$ is a concave function in $\Gamma$,}
\ee
and
\be\label{in.4}
f>0\,\,\mbox{in $K$,}\,\, f=0\,\,\mbox{on $\partial\Gamma,$}
\ee
where $\Gamma\subset\R^n$ is an open symmetric convex cone such that
\be\label{in.5}
\Gamma^+_n:=\{\lambda\in\R^n:\,\,\mbox{each component $\lambda_i>0$}\}\subset\Gamma.
\ee
Moreover, we shall assume that $f$ is normalized
\be\label{in.6}
f(1, \cdots, 1)=1
\ee
and satisfies the more technical assumptions
\be\label{in.7}
\mbox{$f$ is homogeneous of degree one}.
\ee
In addition, for every $C>0$ and every compact set $\Gamma_0$ in $\Gamma$ there is a number
$R=R(C, \Gamma_0)$ such that
\be\label{in.8}
f(\kp_1, \cdots, \kp_n+R)\geq C\,\,\mbox{for all $\kp\in \Gamma_0.$}
\ee

An example of a function satisfying all of these assumptions above is given by
$f=(H_k/H_l)^{\frac{1}{k-l}},$ $0\leq l<k,$ where $H_l$ is the normalized $l$-th elementary symmetric polynomial.
(e.g, $H_0=1, H_1=H, H_n=K$ the extrinsic Gauss curvature.) In \cite{CD}, complete non-compact strictly convex
$Q_k$ flow was studied by Choi and Daskalopoulos, here we don't need the convexity assumption only \textit{admissible}
(see Definition \ref{predf.1}).

Since $f$ is symmetric, from \eqref{in.3}, \eqref{in.6}, and \eqref{in.7} we have
\be\label{in.9}
f(\lambda)\leq f(1)+\sum f_i(1)(\lambda_i-1)=\frac{1}{n}\sum\lambda_i\,\,\mbox{in $\Gamma$}
\ee
and
\be\label{in.10}
\sum f_i(\lambda)=f(\lambda)+\sum f_i(\lambda)(1-\lambda_i)\geq f(1)=1\,\,\mbox{in $\Gamma$}.
\ee

Following \cite{SS}, we shall consider general curvature flow for graphs defined on a relatively open set
\be\label{in.11}
\Omega\equiv\bigcup\limits_{t\geq 0}\Omega_t\times\{t\}\subset\R^{n+1}\times[0, \infty).
\ee
Our main results are the following:
\begin{theorem}
\label{inth.1}
Let $A\subset\R^{n+1}$ be a bounded open set and $u_0:A\goto\R$ a locally Lipschitz continuous function
with $u_0(x)\goto\infty$ for $x\goto x_0\in\partial A,$ moreover, $u_0$ is weakly admissible. Then there exists
$(\Omega, u),$ where $\Omega\subset\R^{n+1}\times[0, \infty)$ is relatively open, such that u solves graphical
general curvature flow equation
\be\label{in.12}
\dot{u}=\sqrt{1+|Du|^2}F\lt(\frac{\gamma^{ik}u_{kl}\gamma^{lj}}{W}\rt)\,\,\mbox{in $\Omega\setminus(\Omega_0\times\{0\})$}
\ee
where $F(S)=f(\kappa[S])$ and $S$ is a symmetric $n\times n$ matrix. Furthermore, the function
$u$ is smooth for $t>0$ and continuous up to $t=0,$ $\Omega_0=A,$ $u(\cdot, 0)=u_0$ in $A$ and $u(x, t)\goto\infty$
as $(x, t)\goto\partial\Omega,$ where $\partial\Omega$ is the relative boundary of $\Omega$ in $\R^{n+1}\times[0, \infty).$
\end{theorem}

A direct consequence of the existence of this smooth solution is the existence of the weak flow.
\begin{theorem}
\label{inth.2}
Let $(A, u_0)$ and $(\Omega, u)$ be as in Theorem \ref{inth.1}. Assume that the level set evolution of $\partial\Omega_0$
doesn't fatten, then it coincides with $(\partial^\mu\Omega_t)_{t\geq 0},$
where $\partial^\mu\Omega_t:=\{x\in\R^{n+2}: \forall r>0, 0<|\Omega_t\cap B_r(x)|<|B_r(x)|\}.$
\end{theorem}
\begin{remark}
\label{inrmk.1}
One can show that if $\partial\Omega_0$ is a embedded Lipschitz hypersurface, then in a short time, the general curvature flow is not fattening.
However, in general we don't know whether the solutions $(\Omega, u)$ are level set solutions.
\end{remark}

For being self-contained, we will include some expository illustrations here, one can find illustrations with more details in \cite{SS}.
Consider a function $u_0:\Omega_0\goto\R,$ for example $u_0(x):=\frac{1}{\text{dist}(x, \partial\Omega_0)}+|x|^2,$ and $u_0$ is \textit{admissible}.
Then by the existence theorem, we obtain the weak evolution of $\partial\Omega_0$, which is defined by the boundary of the domain of $u(x, t).$

\[\def\svgwidth{3 in}
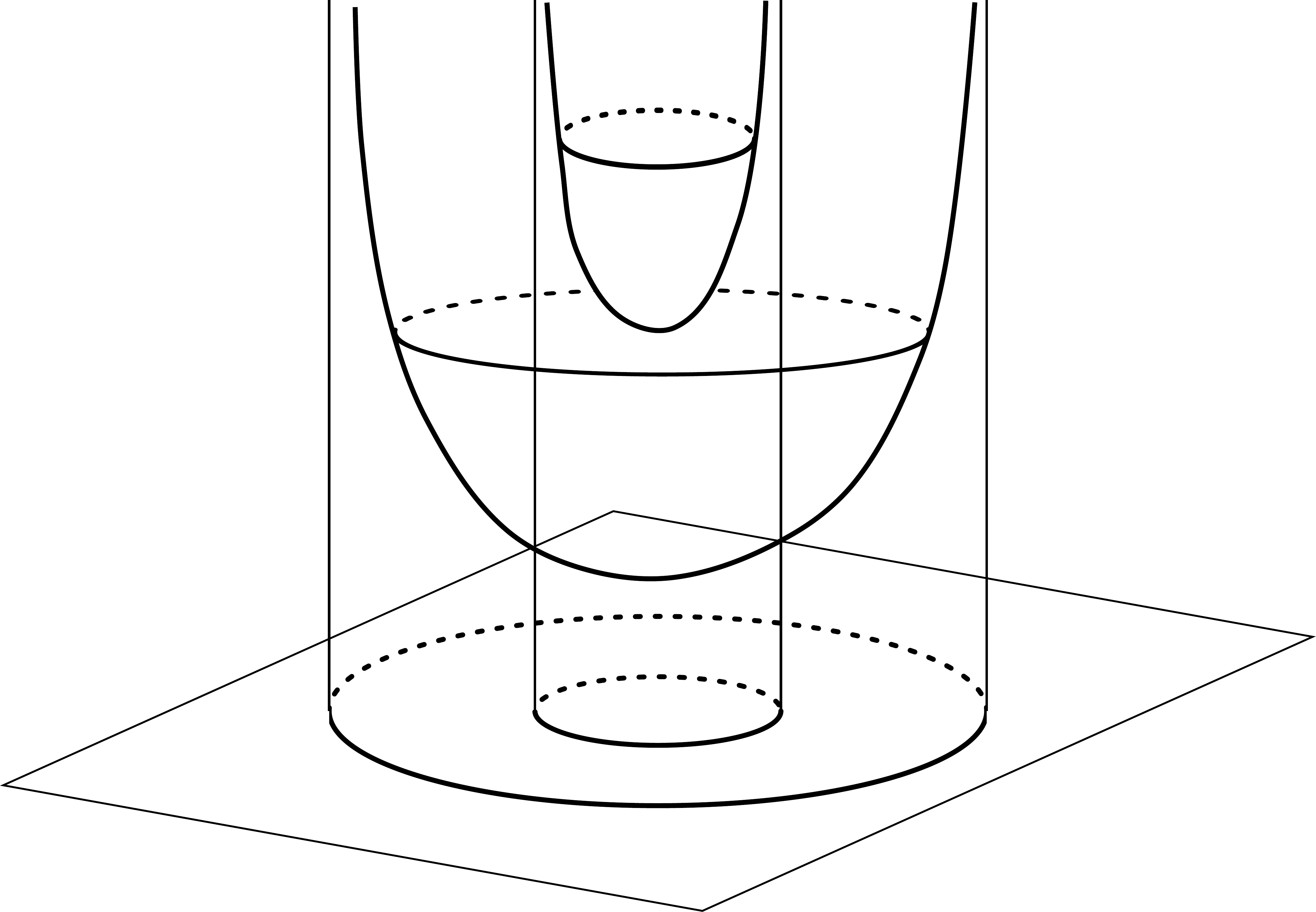\]
\[\mbox{Figure 1. Graph over a ball}\]

In Figure 1, we study the evolution of a graph over $B_1(0),$ and $\partial B_1(0)=\partial\Omega_0.$ Since our general curvature $f$ satisfies \eqref{in.7},
the flow behaves just like the mean curvature flow. One can see that our initial surface $u_0$ is asymptotic to the cylinder $\mathbb{S}^n\times\R.$
Moreover, as we proved in Section \ref{ex}, $u(x, t)$ continues to be asymptotic to the evolving cylinder, which contracts in finite time. However,
our graph $u(x, t)$ does not become singular, but disappears to infinity at the time the cylinder contracts. Note that near the singular time the lowest point covers arbitrarily large distance in arbitrarily small time intervals.

Figure 2 illustrates a graph over a set that develops a "neck-pinch" at $t=T.$ As $t\nearrow T,$ the graph splits above the "neck-pinch"
into two disconnected components without becoming singular. The rest of the evolution is similar to the situation above.
\[\def\svgwidth{5 in}
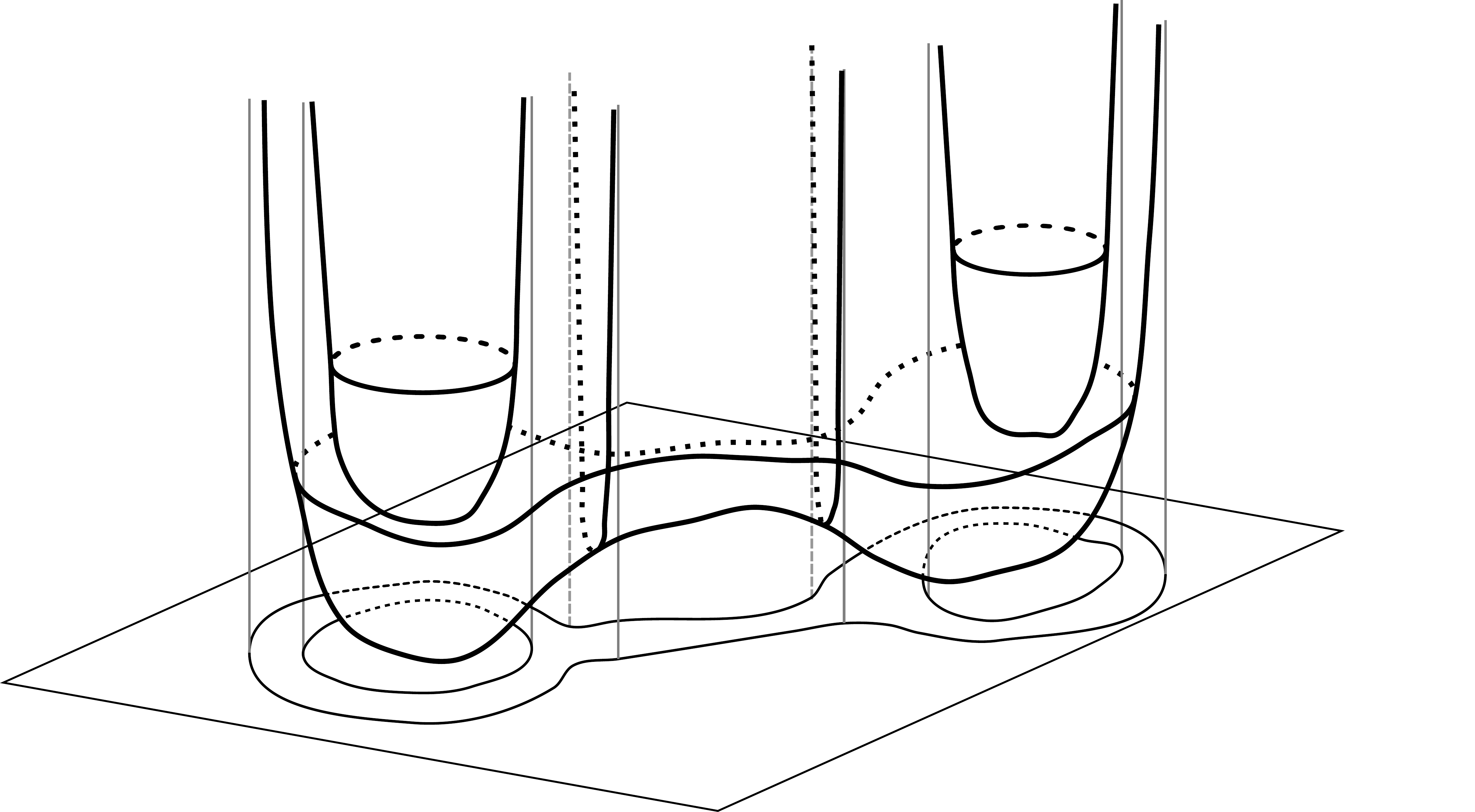\]
\[\mbox{Figure 2. Graph over a set that develops a "neck-pinch"}\]

This paper is organized as following. In Section \ref{pre}, we give some basic definitions and Lemma.
Section \ref{ie} gives interior estimates we need for proving the existence result in Section \ref{ex}.
The key ingredient for proving the existence result is the existence of general curvature flow with boundary conditions, which is proved in
Section \ref{ea}.

\bigskip

\section{Preliminary}
\label{pre}
\setcounter{equation}{0}
For reader's convenience, we will give some definitions first. One can also find most of those definitions
in \cite{SS}.
\begin{defin}
\label{predf.1}
(i)\textit{\textbf{Domain of definition}}: Let $\Omega\subset\R^{n+1}\times[0, \infty)$
be a (relatively) open set. Set
\[\Omega_t:\pi_{\R^{n+1}}(\Omega\cap(\R^{n+1}\times\{t\})),\]
where $\pi_{\R^{n+1}}:\R^{n+2}\goto\R^{n+1}$ is the orthogonal projection onto the first
$n+1$ components. Note that the first $n+1$ components on the domain $\Omega$
are spatial, while the last component can be understood as the time component $t.$\\
(ii)\textit{\textbf{The solution}}: A function $u:\Omega\goto\R$ is called a classical solution to graphical
general curvature flow in $\Omega$ with locally Lipschitz continuous initial value
$u_0:\Omega_0\goto\R,$ if $u\in C^{2,1}_{loc}(\Omega\setminus(\Omega_0\times\{0\}))\cap C^0_{loc}(\Omega)$
satisfies
\be\label{pre.1}
\lt\{\begin{aligned}
\dot{u}&=\sqrt{1+|Du|^2}F\lt(\frac{\gamma^{ik}u_{kl}\gamma^{lj}}{\sqrt{1+|Du|^2}}\rt)\,\,
\mbox{in $\Omega\setminus(\Omega_0\times\{0\})$}\\
u(\cdot, 0)&=u_0\,\,\mbox{in $\Omega_0$},
\end{aligned}\right.
\ee
and $F\lt(\frac{\gamma^{ik}u_{kl}\gamma^{lj}}{\sqrt{1+|Du|^2}}\rt)>0$ in $\Omega\setminus(\Omega_0\times\{0\}),$
where $u_0$ is weakly admissible.\\
(iii)\textit{\textbf{Maximality condition}}: A function $u: \Omega\goto\R$ fulfills the maximality condition
if $u\geq c$ for some $c\in \R$ and if $u|_{\Omega\cap(\R^{n+1}\times[0, T])}$ is proper for every $T>0.$

An initial value $u_0:\Omega\goto\R,\,\,\Omega_0\subset\R^{n+1},$ is said to fulfill the maximality condition if
$\omega:\Omega_0\times[0,\infty)\goto\R$ defined by $\omega(x, t):=u_0(x)$ fulfills the maximality condition.\\
(iv)\textit{\textbf{Admissible condition}}: A function $u: \Omega\goto \R$ is admissible, if for every $t>0,$
$f(\kp[\Sigma_t])>0,$ where $\Sigma_t=\{(x, u(x, t))|x\in\Omega_t\}.$

An initial value $u_0:\Omega\goto\R,\,\,\Omega_0\subset\R^{n+1},$ is said to be admissible if
$\omega:\Omega_0\times[0,\infty)\goto\R$ defined by $\omega(x, t):=u_0(x)$ is admissible. Moreover,
an initial value $u_0:\Omega_0\goto\R,\,\,\Omega_0\subset\R^{n+1},$ is said to be \textit{\textbf{weakly admissible}},
if it can be approached by a sequence of admissible functions, i.e. there exists a sequence $\{u_i\}_{i=1}^\infty$
uniformly converges to $u_0$ on every compact set $K\subset\Omega_0,$ where $u_i$ is admissible and $|D^m u_i|<C=C(K, m),$
$m\in\mathbb{N}.$\\
(iv)\textit{\textbf{Singularity resolving solution}}: A function $u:\Omega\goto\R$
is called a singularity resolving solution to the general curvature flow in dimension $n$
with initial value $u_0:\Omega_0\goto\R$ if\\
(a) $\Omega$ and $\Omega_0$ are as in (i);\\
(b) $u$ is a classical solution to graphical general curvature flow with initial value $u_0$ as in (ii)
and $u$ fulfills the maximality condition.
\end{defin}

\begin{remark}
\label{prermk.1}
The maximality condition implies that $u$ tends to infinity if we approach a point in the relative boundary $\partial\Omega.$
It also ensures that $u(x, t)$ tends to infinity as $|x|$ tends to infinity.
\end{remark}

Next, we will list some evolution equations that will be used later. Since the calculations are straightforward, we will only
state our results here.
\begin{lemma}
\label{prelm.1}
Let $X$ be a solution to the general curvature flow \eqref{in.1}. Then we have the following evolution equations:\\
(i)$\lt(\frac{d}{dt}-F^{ij}\nabla_{ij}\rt)u=0$
where $u=\lt<X, e^{n+1}\rt>,$\\
(ii)$\frac{d}{dt}g_{ij}=-2Fh_{ij},$\\
(iii)$\frac{d}{dt}g^{ij}=2Fh^{ij},$\\
(iv)$\frac{d}{dt}\nu=-g^{ij}F_i\tau_j,$\\
(v)$\lt(\frac{d}{dt}-F^{ij}\nabla_{ij}\rt)\nu^{n+1}=F^{ij}h^k_ih_{kj}\nu^{n+1}$
where $\nu^{n+1}=\lt<\nu, e^{n+1}\rt>,$\\
(vi)$\frac{d}{dt}h^j_i=F^j_i+Fh^k_ih^j_k$ where $h^j_i=g^{jk}h_{ki},$\\
(vii)$\frac{d}{dt}F=F^{ij}F^j_i+FF^{ij}h^k_ih^j_k,$\\
(viii)$\lt(\frac{d}{dt}-F^{ij}\nabla_{ij}\rt)w=-wF^{ij}h^k_ih_{kj}-\frac{2}{w}F^{ij}w_iw_j$
where $w=(\nu^{n+1})^{-1}.$
\end{lemma}

\bigskip
\section{Interior estimates}
\label{ie}
\setcounter{equation}{0}

In order to obtain the interior estimates, in the following, we let
$\varphi=(M-u)_+$ be a cutoff function, where $M>0$ is a constant. Without loss of generality, in this section
we assume $u_0$ in equation \eqref{pre.1} to be smooth and admissible.

\begin{theorem}($C^1$ interior estimates)
\label{ieth.1}
Let $u$ be an admissible solution of equation \eqref{pre.1} and $u$ satisfies the maximality condition. Then
\be\label{ie.1}
\sup w\vp^2\leq\sup\limits_{\{p\in\Omega_0|u(p,0)<M\}}w(p, 0)\vp^2(p, 0)
\ee
\end{theorem}
\begin{proof}
By Lemma \ref{prelm.1} we have
\be\label{ie.2}
\begin{aligned}
L(w\vp^2)&=\lt(\frac{d}{dt}-F^{ij}\nabla_{ij}\rt)(w\vp^2)\\
&=\vp^2Lw+2w\vp L\vp-2F^{ij}w\vp_i\vp_j-2\vp F^{ij}w_i\vp_j\\
&=-w\vp^2 F^{ij}h^k_ih_{kj}-\frac{2\vp^2}{w}F^{ij}w_iw_j-2wf^{ij}\vp_i\vp_j-4\vp F^{ij}w_i\vp_j.\\
\end{aligned}
\ee
If $w\vp^2$ achieves its maximum at an interior point, then at this point we have
\be\label{ie.3}
w_i\vp^2+2\vp w\vp_i=0,
\ee
therefore,
\be\label{ie.4}
L(w\vp^2)=-w\vp^2F^{ij}h^k_ih_{kj}-2wF^{ij}\vp_i\vp_j<0.
\ee
By the maximum principle, it's impossible. Thus, we proved this theorem.
\end{proof}

\begin{theorem}(Lower bound of speed)
\label{ieth.2}
Let $\Sigma_t=\{(x, u(x, t)| x\in\Omega_t\}$ be a smooth graph of $u$ with positive $F$ curvature,  where $u$ is an admissible solution of
equation \eqref{pre.1}and $u$ satisfies the maximality condition. Then for any $t>0$ we have
\be\label{ie.5}
\vp^{-1}(p,t)F\geq\inf\limits_{\{p\in\Omega_0|u(p, 0)<M\}}\vp^{-1}(p, 0)F(p, 0).
\ee
\end{theorem}
\begin{proof}
By Lemma \ref{prelm.1} we have
\be\label{ie.6}
\begin{aligned}
L(\vp^{-1}F)&=\vp^{-1}LF-\vp^{-2}FL\vp-2F\vp^{-3}F^{ij}\vp_i\vp_j+2\vp^{-2}F^{ij}\vp_iF_j\\
&=\vp^{-1}FF^{ij}h^k_ih^j_k+2\vp^{-1}F^{ij}\vp_j(\vp^{-1}F)_i.\\
\end{aligned}
\ee
If $\vp^{-1}F$ achieves its minimum at an interior point, then at this point we would have
\[L(\vp^{-1}F)>0,\]
which contradicts to the maximum principle. Therefore, we proved this theorem.
\end{proof}

Finally, we will follow the idea of \cite{SUW} and prove the $C^2$ interior estimates.
\begin{theorem}($C^2$ interior estimates)
\label{ieth.3}
Suppose u is an admissible solution of equation \eqref{pre.1} and $u$ satisfies the maximality condition.
Then when $u<M$ we have
\be\label{ie.7}
|D^2u|\leq\frac{C}{(M-u)},
\ee
where $C=C(u_0).$
\end{theorem}
\begin{proof}
First by Theorem \ref{ieth.1}, we know that there exists $a=a(M+1)>0$ such that
$\nu^{n+1}\geq 2a>0$ when $u\leq M.$

We consider function $\psi=\beta\log\vp+\log\kp_{\max}-\log(\nu^{n+1}-a),$ where $\beta>0$ to be determined.
Assume $\psi$ achieves its maximum at an interior point $X_0=(x_0, u(x_0, t_0)).$
We can choose a local coordinates in the neighborhood of $X_0$ such that at $X_0$ we have
$\kappa_1=\kp_{\max}(X_0),$ $g_{ij}=\delta_{ij},$ and $h_{ij}=\kp_i\delta_{ij}.$ We will also assume
$\kp_1\geq\kp_2 \geq \cdots \geq\kp_n.$ Consequently,
we also have $F^{ij}=f_i\delta_{ij}$ at this point. In the following, all calculations are done at this point,
so we will not distinguish between $h^i_j$ and $h_{ij}.$

Now, differentiating $\psi$ at $X_0$ we get
\be\label{ie.8}
\beta\frac{\vp_i}{\vp}+\frac{h_{11i}}{h_{11}}-\frac{\nu^{n+1}_i}{\nu^{n+1}-a}=0.
\ee
Moreover,
\be\label{ie.9}
\begin{aligned}
\frac{d}{dt}\psi&=\beta\frac{\dot{\vp}}{\vp}+\frac{\dot{h_{11}}}{h_{11}}-\frac{\dot{\nu^{n+1}}}{\nu^{n+1}-a}\\
&=-\frac{\beta}{\vp}F^{ii}u_{ii}+\frac{1}{h_{11}}\lt\{F^{ii}h_{ii11}+F^{ij, rs}h_{ij1}h_{rs1}+Fh_{11}^2\rt\}
+\frac{1}{\nu^{n+1}-a}F_iu_i.
\end{aligned}
\ee
In the Euclidean space, we have
\[h_{11ii}-h_{ii11}=h_{11}^2h_{ii}-h_{ii}^2h_{11}.\]
Therefore,
\be\label{ie.10}
\begin{aligned}
\lt(\frac{d}{dt}-F^{ii}\nabla_{ii}\rt)\psi&=\frac{-\beta}{\vp}F^{ii}u_{ii}
\frac{1}{h_{11}}\{F^{ii}h_{ii11}+F^{ij,rs}h_{ij1}h_{rs1}+Fh^2_{11}\}\\
&+\frac{1}{\nu^{n+1}-a}F_iu_i
-F^{ii}\left\{\frac{-\beta}{\vp}u_{ii}-\frac{\beta}{\vp^2}u_i^2+\frac{h_{11ii}}{h_{11}}
-\frac{h_{11i}^2}{h_{11}^2}\right.\\
&\left.-\frac{\nu^{n+1}_{ii}}{\nu^{n+1}-a}+\frac{(\nu^{n+1}_i)^2}{(\nu^{n+1}-a)^2}\right\}\\
&=-\frac{a}{\nu^{n+1}-a}\sum f_i\kp_i^2+\frac{1}{h_{11}}F^{ij, rs}h_{ij1}h_{rs1}+\frac{\beta}{\vp^2}\sum f_iu_i^2\\
&+\frac{1}{\kp_1^2}\sum f_ih^2_{11i}-\sum f_i\frac{(\nu^{n+1}_i)^2}{(\nu^{n+1}-a)^2}.
\end{aligned}
\ee

We will divide it into 2 cases.

Case 1. $\kp_n<\theta\kp_1$ at $X_0.$
We have
\be\label{ie.11}
\begin{aligned}
\lt|\frac{h_{11i}}{h_{11}}\rt|^2&=\lt|\frac{\nu^{n+1}_i}{\nu^{n+1}-a}-\beta\frac{\vp_i}{\vp}\rt|^2\\
&=\lt|\frac{\nu^{n+1}_i}{\nu^{n+1}-a}+\beta\frac{u_i}{\vp}\rt|^2\\
&\leq(1+\e)\lt|\frac{\nu^{n+1}_i}{\nu^{n+1}-a}\rt|^2+(1+\e^{-1})\beta^2\lt|\frac{u_i}{\vp}\rt|^2.\\
\end{aligned}
\ee
Therefore,
\be\label{ie.12}
L\psi\leq\frac{-a}{\nu^{n+1}-a}\sum f_i\kp_i^2+[\beta+(1+\e^{-1})\beta^2]\sum\frac{f_iu_i^2}{\vp^2}
+\e\sum\frac{f_iu_i^2\kp_i^2}{(\nu^{n+1}-a)^2}.
\ee
Since \[\sum f_i\kp_i^2\geq f_n\kp_n^2\geq\frac{\theta^2}{n}\kp_1^2\mathcal{T},\]
where $\mathcal{T}=\sum f_i.$
We get
\be\label{ie.13}
-\frac{c_1\theta^2}{n}\kp_1^2\mathcal{T}+\frac{c}{\vp^2}\mathcal{T}\geq 0
\ee
at $X_0,$
which implies
\be\label{ie.14}
\kp_1^2\vp^2\leq C.
\ee
Case 2. $\kp_n\geq -\theta\kp_1$ at $X_0.$
Let's partition $\{1, \cdots, n\}$ into two parts:
\[I=\{j: f_j\leq 4 f_1\}\,\,\mbox{and}\,\, J=\{j: f_j>4f_1\}.\]
For $i\in I$ we have
\be\label{ie.15}
\begin{aligned}
\frac{1}{\kp_1^2}f_i|\nabla_ih_{11}|^2&=f_i\lt|\frac{\nu^{n+1}_i}{\nu^{n+1}-a}+\beta\frac{u_i}{\vp}\rt|^2\\
&\leq(1+\e)f_i\lt|\frac{\nu^{n+1}_i}{\nu^{n+1}-a}\rt|^2+\frac{c}{\vp^2}(1+\e^{-1})\beta^2f_1.
\end{aligned}
\ee
For $j\in J$ we have
\be\label{ie.16}
\begin{aligned}
\beta f_j\frac{|\nabla_j\vp|^2}{\vp^2}&=\beta^{-1}f_j\lt(\frac{\nu^{n+1}_j}{\nu^{n+1}-a}-\frac{\nabla_jh_{11}}{h_{11}}\rt)^2\\
&\leq\frac{1+\e}{\beta}f_j\frac{|\nabla_j\nu^{n+1}|^2}{(\nu^{n+1}-a)^2}+\frac{1+\e^{-1}}{\beta}f_j\frac{|\nabla_jh_{11}|}{h_{11}^2}.\\
\end{aligned}
\ee
Combining \eqref{ie.15} and \eqref{ie.16} we get
\be\label{ie.17}
\begin{aligned}
&\beta\sum\limits_{i=1}^nf_i\frac{|\nabla_i\vp|^2}{\vp^2}+\sum\limits_{i=1}^nf_i\frac{|\nabla_ih_{11}|}{h_{11}^2}\\
&\leq[\beta+(1+\e^{-1})\beta^2]\sum\limits_{i\in I}f_i\frac{|\nabla_i\vp|^2}{\vp^2}
+(1+\e)\sum\limits_{i\in I}f_i\lt|\frac{\nabla_i\nu^{n+1}}{\nu^{n+1}-a}\rt|^2\\
&+\frac{1+\e}{\beta}\sum\limits_{j\in J}f_j\lt|\frac{\nabla_j\nu^{n+1}}{\nu^{n+1}-a}\rt|^2
+[1+(1+\e^{-1})\beta^{-1}]\sum\limits_{j\in J}f_j\frac{|\nabla_jh_{11}|^2}{h_{11}^2}\\
&\leq 4n[\beta+(1+\e^{-1})\beta^2]f_1\frac{|\nabla\vp|^2}{\vp^2}
+(1+\e)(1+\beta^{-1})\sum\limits_{i=1}^nf_i\lt|\frac{\nu^{n+1}_i}{\nu^{n+1}-a}\rt|^2\\
&+[1+(1+\e^{-1})\beta^{-1}]\sum\limits_{j\in J}f_j\frac{|\nabla_jh_{11}|}{h_{11}^2}.\\
\end{aligned}
\ee
Therefore, at $X_0$ we get
\be\label{ie.18}
\begin{aligned}
0&\leq-\frac{a}{\nu^{n+1}-a}\sum f_i\kp_i^2+\frac{1}{\kp_1}F^{ij, rs}h_{ij1}h_{rs1}
+4n[\beta+(1+\e^{-1})\beta^2]f_1\frac{|\nabla\vp|^2}{\vp^2}\\
&+[(1+\e)(1+\beta^{-1})-1]\sum f_i\lt|\frac{\nu^{n+1}_i}{\nu^{n+1}-a}\rt|^2
+[1+(1+\e^{-1})\beta^{-1}]\sum\limits_{j\in J}f_j\frac{|\nabla_jh_{11}|^2}{h_{11}^2}.\\
\end{aligned}
\ee
Since
\be\label{ie.19}
\begin{aligned}
\frac{1}{\kp_1}F^{ij, rs}h_{ij1}h_{rs1}&\leq\frac{2}{\kp_1}\sum\limits_{j\in J}\frac{f_1-f_j}{\kp_1-\kp_j}|\nabla_jh_{11}|^2\\
&\leq\frac{-3}{2\kp_1}\sum\limits_{j\in J}\frac{f_j}{\kp_1}|\nabla_jh_{11}|^2\\
&=\frac{-3}{2}\sum\limits_{j\in J}\frac{f_j}{\kp_1^2}|\nabla_jh_{11}|^2,
\end{aligned}
\ee
and
\be\label{ie.20}
f_i\lt|\frac{\nu^{n+1}_i}{\nu^{n+1}-a}\rt|^2\leq\frac{1}{a^2}\sum f_i\kp_i^2u_i^2\leq\frac{c_0}{a^2}\sum f_i\kp_i^2.
\ee
We can choose $\beta$ large and $\e>0$ small such that
\[\frac{1+\e^{-1}}{\beta}\leq\frac{1}{2}\,\,\mbox{and}\,\,[(1+\e)(1+\beta^{-1})-1]\frac{c_0}{a^2}\leq \frac{a}{2(1-a)},\]
then we have
\be\label{ie.21}
0\leq-\frac{a}{2(1-a)}f_1\kp_1^2+4n[\beta+(1+\e^{-1})\beta^2]\frac{cf_1}{\vp^2},
\ee
which implies $\vp^2\kp_1^2\leq C.$
Combining case 1 and case 2 we proved this theorem.
\end{proof}

\bigskip
\section{Existence of approximating solutions}
\label{ea}
\setcounter{equation}{0}

Following \cite{SS}, we choose a smooth monotone approximation $g$ of $\min\{\cdot, 0\}$
such that $g(x)=\min\{x, 0\}$ for $|x|>1,$ $0\leq g'\leq 1,$ and set
$\min_\e\{a, b\}:=\e g\lt(\frac{1}{\e}(a-b)\rt)+b.$
We will set $\min_\e\{u(x), L\}:=L$ at $x$ if $u$ is not defined at $x.$

We will prove the following Theorem.
\begin{theorem}
\label{eath.1}
Let $A\subset\R^{n+1}$ be an open set. Assume that $u_0: A\goto\R$ is locally
Lipschitz continuous, weakly admissible, and satisfies maximality condition.
Let $L>0,$ $R>0,$ then there exists a smooth solution
$u_{i, R}^L\in C^{\infty}(B_R\times(0, \infty))$ to
\be\label{ea.1}
\left\{\begin{aligned}
\dot{u}&=\sqrt{1+|Du|^2}F\lt(\frac{\gamma^{ik}u_{kl}\gamma^{lj}}{w}\rt)\,\,
&\mbox{in $B_R(0)\times[0, \infty)$}\\
u&=L\,\,&\mbox{on $\partial B_R(0)\times[0,\infty)$}\\
u(\cdot, 0)&=min_\e\{u_{0,i}, L\}\,\,&\mbox{in $B_R(0)$},\\
\end{aligned}\right.
\ee
where $u_{0, i}$ is a smooth, admissible function and $u_{0, i}\goto u_0$
uniformly on compact set. We always assume that $R\geq R_0(L, i)$ is so large that $u_{0, i}\geq L+1$
on $\partial B_R(0)$ and $\min_\e\{u_{0,i}, L\}$ is weakly admissible in $B_R(0).$
\end{theorem}

Instead of studying equation \eqref{ea.1}, we will study the
following equation:
\be\label{ea.2}
\left\{\begin{aligned}
\dot{u}&=\sqrt{1+|Du|^2}F\lt(\frac{\gamma^{ik}u_{kl}\gamma^{lj}}{w}\rt)\,\,
&\mbox{in $\Omega\times[0, t_0)$}\\
u&=0\,\,&\mbox{on $\partial \Omega\times[0,t_0)$}\\
u(\cdot, 0)&=\tu_0\,\,&\mbox{in $\Omega$},\\
\end{aligned}\right.
\ee
where $\Omega$ is a convex domain, and $\tu_0$ is a smooth, admissible function defined in $\Omega$
moreover, $\tu_0|_{\partial\Omega}=0.$

First, let's recall the well known short time existence theorem (one can find it in \cite{LX1}).
\begin{theorem}
\label{eath.2}
Let $G(D^2u, Du, u)$ be a nonlinear operator that is smooth with respect to $D^2u,$ $Du,$ and $u.$
Suppose that $G$ is defined for a function $u$ belonging to an open set $\Lambda\subset C^2(\Omega)$
and $G$ is elliptic for any $u\in \Lambda,$ i.e., $G^{ij}>0,$ then the initial value problem
\be\label{ea.3}
\left\{\begin{aligned}
u_t&=G(D^2u, Du, u)\,\,&\mbox{in $\Omega\times[0, T^*)$}\\
u(x, t)&=0\,\,&\mbox{on $\partial\Omega\times[0, T^*)$}\\
u(x, 0)&=u_0\,\,&\mbox{in $\Omega\times\{0\}$}\\
\end{aligned}\right.
\ee
has a unique smooth solution $u$ when $T^*=\e>0$ small enough, except for the corner,
where $u_0\in \Lambda$ be of class $C^{\infty}(\bar{\Omega}).$
\end{theorem}

Our strategy is to obtain estimates on $[0, t_0).$  Then we can repeat
the process and obtain the long time existence result.

\begin{lemma}
\label{ealm.1}
Let $u$ be a solution of equation \eqref{ea.2} and $u(\cdot, t)$ is admissible for $t\in[0, t_0)$. Then
\[w=\sqrt{1+|Du|^2}\leq C\,\,\mbox{in $\bar{B}_R(0)\times[0, t_0)$}\]
where $C=C(\tu_0, \Omega)$
\end{lemma}
\begin{proof}
By Lemma \ref{prelm.1} item (viii) and the maximum principle we have
\be\label{ea.4}
w\leq\max\{\sup\limits_{t=0}w(\cdot, 0), \max\limits_{\partial\Omega\times[0, t_0)}w\}.
\ee
Since $\Omega$ is convex, by a standard process (see \cite{GT}), we can construct a function $\underline{u}$ in $\Omega$
such that
\[F\lt(\frac{\gamma^{ik}\underline{u}_{kl}\gamma^{lj}}{w}\rt)\geq F\lt(\frac{\gamma^{ik}\tu_{0kl}\gamma^{lj}}{w}\rt)\]
and $\underline{u}|_{\partial\Omega}=0.$
Applying the maximum principle again and we obtain
\be\label{ea.5}
0\geq u(x, t)\geq\tu_0\geq\underline{u}(x)\,\,\mbox{in $\Omega\times[0, t_0)$},
\ee
thus $|\nabla u|_{\partial\Omega}\leq C,$ proved the Lemma.
\end{proof}

Next, we will derive a uniform upper bound for the curvature $F.$
\begin{lemma}
\label{ealm.2}
Let $u$ be a solution of equation \eqref{ea.2}. Then we have
\[0<F\leq C\,\,\mbox{in $\Omega\times[0, t_0)$}\]
where $C=C(\tu_0, |Du|).$
\end{lemma}
\begin{proof}
The positivity of $F$ is a direct consequence of Lemma \ref{prelm.1} item (vii) and the maximum priciple.
We want to show $F\leq C.$
Let's consider $\psi=\log F-\log (\nu^{n+1}-a),$ where $\nu^{n+1}\geq 2a>0$ in $\bar{\Omega}\times[0, t_0).$
If $\psi$ achieves its maximum at an interior point $(x^*, t^*)$ then by Lemma \ref{prelm.1}, at this point we would have
\be\label{ea.6}
\lt(\frac{d}{dt}-F^{ij}\nabla_{ij}\rt)\psi=\frac{-a}{\nu^{n+1}-a}\sum f_i\kp_i^2<0
\ee
leads to a contradiction.
Therefore we conclude that
\be\label{ea.7}
\max\frac{F}{\nu^{n+1}-a}=\max\limits_{t=0}\frac{F}{\nu^{n+1}-a},
\ee
which implies the Lemma.
\end{proof}

\begin{lemma}
\label{ealm.3}
Let $u$ be a solution of equation \eqref{ea.2}, and $u(\cdot, t)$ is admissible. Then
\be\label{ea.8}
|D^2u(\cdot, t)|\leq C,
\ee
where $C=C(\tu_0, |Du|, |D^2u(\cdot, t)|_{\partial\Omega}).$
\end{lemma}

This Lemma can be proved by a similar calculation as in the proof of Theorem \ref{ieth.3}.  The only diffenrence is
we need to choose $\beta=0$ here, so we will skip the proof.

Therefore, in order to get the $C^2$ estimate, we only need to prove the following:
\begin{lemma}
\label{ealm.4}
Let $u$ be a solution of equation \eqref{ea.2} and $u(\cdot, t)$ is admissible. Then we have
\be\label{ea.9}
|D^2u(\cdot, t)|_{\partial\Omega}\leq C\,\,\mbox{for $t\in[0, t_0)$}
\ee
where $C=C(\tu_0).$
\end{lemma}
In order to prove Lemma \ref{ealm.4}, we will need to prove the following proposition first, the idea goes back to \cite{CNS5}.

\begin{proposition}
\label{eapr.1}
Let $u$ be a solution of equation \eqref{ea.2} and $u(\cdot, t)$ is admissible. Then given any $\eta>0$
there exists $\delta_0=\delta_0(\tu_0)>0$ such that in a $\delta_0$ neighborhood of $\partial\Omega$ we have
$u_n\leq\eta$ for any $t\in[0, t_0),$ where $u_n$ is the derivative of $u$ in the interior normal direction.
\end{proposition}

Suppose in the following that the origin belongs to $\partial\Omega$ and that the $x_n$-axis
is the interior normal.
\begin{proof}
Without loss of generality, we may assume that the graph $\Sigma_0=\{(x, \tu_0(x))|x\in\Omega\}$
lies in the ball $B_{1/2}$ with center at $(0, 0, \cdots, 1/2, 0)$ in $\R^{n+1}.$
Then since
\[u_t=\frac{1}{w}F\lt(\frac{\gamma^{ik}u_{kl}\gamma^{lj}}{w}\rt)>0\,\,\mbox{in $\Omega\times[0, t_0)$}\]
and $u(x, t)\leq 0$ we get
\[\Sigma_t=\{(x, u(x, t))|x\in\Omega\}\,\, t\in[0, t_0)\]
all lie in $B_{1/2}.$ Moreover, by the Hopf Lemma, we have $u_n(0, t)<0$ for $t\in[0, t_0).$

Now, consider a family of reflections $I_\delta$ depending on a parameter $\delta>0,$ in the boundary of
the unit ball in $\R^{n+1}:$ $B_1(e^\delta)=B^\delta,$ with center $e_\delta=(0, \cdots, 0, 1+\delta, \tC\delta),$
where $\tC>1$ to be determined.

To start, $\Sigma_t,$ $0\leq t<t_0$ is contained in $B_1(e^0).$ As $\delta$ becomes positive,
a portion of $\Sigma_t$ near the origin in $\R^{n+1}$ lies outside $B^\delta.$ For a very small $\delta,$
the reflection $I_\delta(\Sigma_t\cap\mathcal{C}B^\delta)$ doesn't touch $\Sigma_t\cap B^\delta.$
Furthermore, at any $X^0\in\Sigma_t\cap\partial B^\delta,$ $I_\delta(\Sigma_t\cap\mathcal{C}B^\delta)$
is not tangent to $\Sigma_t$ for $t\in [0, t_0).$

Suppose there is a first value of $\delta\leq\delta_0=\delta_0(\tu_0)$ for which this statement fails, i.e., there exists
$t^*\in (0, t_0)$ such that either
\[\mbox{(a) $I_\delta(\Sigma_{t^*}\cap\mathcal{C}B^\delta)$ touches $\Sigma_{t^*}$ at a point $I_\delta(X^0)$}\]
or
\[\mbox{(b) $I_\delta(\Sigma_{t^*}\cap\mathcal{C}B^\delta)$ is tangent to $\Sigma_{t^*}$ at some point $X^0\in\partial B^\delta\cap\Sigma_{t^*}$.}\]
We claim: For $\delta\leq\delta_0=\delta_0(u_0),$ $t<t_0,$ both cases are impossible.

If the claim is true, it follows that for $\delta\leq\delta_0,$ $t<t_0,$ a point $X\in\Sigma_t$ belongs to $\partial B^\delta$
then we have
\be\label{ea.10}
(X-e_\delta)\cdot\nu(X)<0.
\ee
In particular, if we take $X=(0, \cdots, 0, x_n, u(0, x_n, t))$ then
\be\label{ea.11}
(x_n-1-\delta)\nu_n(x)+(u-\tC\delta)\nu^{n+1}(x)<0,
\ee
which is equivalent to
\be\label{ea.12}
\frac{(1+\delta-x_n)u_n}{w}+\frac{u-\tC\delta}{w}<0.
\ee
Thus we have,
\be\label{ea.13}
u_n(0, x_n)<2|\tC\delta-u|<\eta
\ee
if $x_n,$ $\delta_0$ are sufficiently small.

Therefore, we only need to prove our claim. We will prove it by contradiction.
Suppose case (a) first occurs at $t^*>0$.
Let's denote
\be\label{ea.14}
\tilde{X}=I_\delta(X)=\frac{X-e_\delta}{|X-e_\delta|^2}.
\ee
By a straightforward calculation we have
\be\label{ea.15}
\begin{aligned}
\dot{\tilde{X}}&=\frac{\dot{X}}{|X-e_\delta|^2}-2\frac{(X-e_\delta)\lt<\dot{X}, X-e_\delta\rt>}{|X-e_\delta|^4}\\
&=\frac{F}{|X-e_\delta|^2}\tilde{\nu},
\end{aligned}
\ee
where $\tilde{\nu}$ is the unit normal of the reflected surface.
Moreover, we can also compute the principle curvature of the reflected surface and get
\be\label{ea.16}
\tilde{\kp}=\kp|X-e_\delta|^2+2\lt<X-e_\delta, \nu\rt>,
\ee
thus
\be\label{ea.17}
f(\tilde{\kp})=f(\kp(X))|X-e_\delta|^2+2\lt<X-e_\delta, \nu\rt>\sum f_i.
\ee
When $t=0,$ since for any $X\in\Sigma_0\cap\mathcal{C}B^\delta$ there is a $\delta'<\delta$
such that $X\in\partial B^{\delta'}$ and
\[(X-e_{\delta'})\cdot\nu(X)<0.\]
Thus
\be\label{ea.18}
\begin{aligned}
2(X-e_\delta)\cdot\nu(X)&=2(X-e_{\delta'})\cdot\nu{X}+2(e_{\delta'}-e_{\delta})\cdot\nu(X)\\
&<2(e_{\delta'}-e_\delta)\cdot\nu(X)\\
&=2\nu_n(\delta'-\delta)+2\nu^{n+1}\tC(\delta'-\delta)\\
&\leq-c\tC(\delta-\delta').
\end{aligned}
\ee
Furthermore,
\be\label{ea.19}
\begin{aligned}
|X-e_\delta|^2&=|X-e_{\delta'}+e_{\delta'}-e_\delta|^2\\
&=1+2(X-e_{\delta'})\cdot(e_{\delta'}-e_{\delta})+|e_{\delta'}-e_\delta|^2\\
&\leq 1+A(\delta-\delta')+\tC(\delta-\delta')(\tC\delta'-u(x))+\tC^2\delta(\delta-\delta')\\
&\leq 1+A(\delta-\delta')+A\tC(\delta-\delta')x_n,
\end{aligned}
\ee
here we used $\tC^2\delta_0\leq 1$ (we can always choose $\delta_0>0$ small such that this inequality holds) and  $-\tu_0(x)\leq Ax_n.$

Since $|X-e_{\delta'}|=1$ we have for $x=(x', x_n)$
\be\label{ea.20}
|x'|^2+(1+\delta'-x_n)^2+|\tC\delta'-\tu_0(x)|^2=1,
\ee
which implies
\be\label{ea.21}
|x'|^2+(1/2-x_n)^2+2(1/2-x_n)(1/2+\delta')+(1/2+\delta')^2+(\tC\delta'+Ax_n)\geq1.
\ee
The sum of the first two terms is at most $1/4,$ hence we get
\be\label{ea.22}
x_n(1+2\delta')\leq A(\delta'+x_n^2)
\ee
which yields
\be\label{ea.23}
x_n\leq A\delta'.
\ee
Plug it back into \eqref{ea.19} we have
\be\label{ea.24}
|X-e_\delta|^2\leq 1+A(\delta-\delta').
\ee
Combine with \eqref{ea.16} we obtain
\be\label{ea.25}
\tilde{\kp}\leq\kp[1+A(\delta-\delta')]-c\tC(\delta-\delta'),
\ee
therefore
\be\label{ea.26}
f(\tilde{\kp})\leq f(\kp)(1+A\tau)-c\tC\tau,
\ee
where $\tau=\delta-\delta'.$
Since $f(\kp(\tilde{x}))\geq f(\kp(x))-B\tau$ for $B$ depends on $|\tu_0|_{C^3}$ under control,
we can choose $\tC$ large depends on $|f|_{C^0}$ such that
$f(\tilde{\kp})<f(\kp(\tilde{x})),$ then we have
at $t=0$ the reflected surface $\tilde{\Sigma}_0=I_\delta(\Sigma_0)$
lies above $\Sigma_0$ for $\delta\leq\delta_0.$

In order to show case $(a)$ doesn't occur, we only need to show
$\tilde{\Sigma}_t$ lies above $\Sigma_t$ for $t\in(0, t_0)$ and there is no interior
touching point. If not, suppose $(a)$ first occur at $X^*=(x^*, u(x^*, t^*)).$
At the point $I_\delta(X^*)$ we have
\be\label{ea.27}
\begin{aligned}
\frac{\dot{\tu}}{w}-f(\tilde{\kp})&=\frac{f(\kp(X^*))}{|X^*-e_\delta|^2}-f(\kp(X^*))|X^*-e_\delta|^2
+c\tC(\delta-\delta')\sum f_i\\
&\geq f(\kp(X^*))\lt(\frac{1}{A\tau}-(1+A\tau)\rt)+c\tC\tau>0
\end{aligned}
\ee
when $\tC$ large under control.
Since
\be\label{ea.28}
\frac{\dot{u}}{w}-f(\kp(\tilde{x^*}))=0
\ee
we have
\be\label{ea.29}
L(\tilde{u}-u)>0\,\,\mbox{at the point $I_\delta(X^*)$}
\ee
by the maximum principle we have a contradiction. Therefore, case $(a)$ doesn't occur.

Now, let's turn to case $(b).$ Same as before, for any $X\in \Sigma_t\cap\mathcal{C}B^\delta$
we can derive $L(\tilde{u}-u)>0.$ Since $\tilde{\Sigma_t}$ lies above $\Sigma_t$ for any $t\in[0, t_0)$,
by the Hopf Lemma, case $(b)$ cannot occur.
\end{proof}

Now we are going to apply Proposition \ref{eapr.1} to prove Lemma \ref{ealm.4}.
\begin{proof} (proof of Lemma \ref{ealm.4})
In the following, we denote $G(D^2u, Du, u_t)=\frac{u_t}{w}-F=0.$
Similar to \cite{CNS3} we can prove
\begin{lemma}
\label{ealm.5}
Suppose $f$ satisfies \eqref{in.2}, \eqref{in.3}, \eqref{in.6}, and \eqref{in.7}.
Then $\mathcal{L}(x_iu_j-x_ju_i)=0,$ $\mathcal{L}u_i=0,$ $1\leq i, j\leq n,$
and $\mathcal{L}u_t=0.$
\end{lemma}

If we let $v(x, t)=\frac{1}{a}u(ax, a^2t),$ then at $a=1$ we have
\be\label{ea.30}
\begin{aligned}
\frac{d}{da}G(D^2v, Dv, v_t)&=\mathcal{L}\lt(\frac{d}{da}v\rt)\\
&=\mathcal{L}(ru_r-u2tu_t)=0,
\end{aligned}
\ee
here $r$ is the polar coordinate.
Therefore we have
\be\label{ea.31}
\mathcal{L}(ru_r-u)=-\mathcal{L}(2tu_t)=-\frac{2u_t}{w}=-2F<0.
\ee
Following \cite{CNS5}, by doing infinitesimal rotation in $\R^{n+1}$ we let
\be\label{ea.32}
v(x, t)=u(x', x_n+u(x, t)d\theta, t)+x_nd\theta+\,\,\mbox{higher order in $d\theta,$}
\ee
and at $x$ we have
\[G(D^2v, Dv, v_t)=0.\]
We compute the first order term in $d\theta$ and get
\be\label{ea.33}
\mathcal{L}(x_n+u(x, t)u_n)=0
\ee

Now consider an arbitrary point on $\partial\Omega,$ which we may assume to be the origin of $\R^n,$
and choose the coordinates so that the positive $x_n$ axis is the interior normal to $\partial\Omega$ at the origin.
We may assume that the boundary near the origin is represented by
\[x_n=\rho(x')=\frac{1}{2}\sum\limits_{\alpha=1}^{n-1}\lambda_\alpha x_\alpha^2+O(|x'|^3)\]
where $\lambda_\alpha>0$ are the principal curvatures of $\partial\Omega$ at the origin.

Since $u(x', \rho(x'), t)\equiv 0$ on $\partial\Omega\times[0, t_0),$ we have at the origin
\[u_{\alpha\beta}+u_n\rho_{\alpha\beta}=0\,\,\mbox{and $u_\alpha+u_n\lambda_\alpha x_\alpha=0$}.\]
Now let $T_\alpha=\partial_\alpha+\lambda_\alpha(x_\alpha\partial_n-x_n\partial_\alpha),$
then we have
\be\label{ea.34}
|T_\alpha u|\leq C \,\,\mbox{in a small neighborhood of the origin,}
\ee
and
\be\label{ea.35}
|T_\alpha u|\leq C|x|^2\,\,\mbox{on $\partial\Omega$ near the origin.}
\ee
Moreover, by Lemma \ref{ealm.5} we have $\mathcal{L}T_\alpha u=0.$
In the following, $\Omega_\beta$ denote a small region $\Omega\cap\{x_n<\beta\}.$
For $\delta,$ $\beta$ small, set
\be\label{ea.36}
h=ru_r-u-\frac{\delta}{\beta}(x_n+uu_n),
\ee
then by \eqref{ea.31} and \eqref{ea.33}
we have
\be\label{ea.37}
\mathcal{L}h<0\,\,\mbox{in $\Omega_\beta.$}
\ee
On the lower bound of $\Omega_\beta,$ since $u=0$ we have
\be\label{ea.38}
|ru_r|\leq C_1|x|^2\,\,\mbox{also $x_n\geq a|x|^2, a>0.$}
\ee
Therefore,
\be\label{ea.39}
\begin{aligned}
h&=ru_r-\frac{\delta}{\beta}x_n\\
&\leq(C_1-\frac{\delta}{\beta}a)|x|^2.
\end{aligned}
\ee
For $\frac{\delta}{\beta}$ and $A$ large we have

\be\label{ea.40}
Ah\leq-C|x|^2\,\,\mbox{on the lower boundary of $\Omega_\beta.$}
\ee
Finally, on $x_n=\beta,$ by gradient estimate Lemma \ref{ealm.1} we have $0\leq -u\leq C\beta.$
Thus,
\be\label{ea.41}
\begin{aligned}
h&=\beta u_n+\sum\limits_{\alpha=1}^{n-1}x_\alpha u_\alpha-u(1+\frac{\delta}{\beta}u_n)-\delta\\
&\leq\frac{3}{2}\eta\beta+C\beta^{1/2}-u(1+\frac{3\delta}{2\beta}\eta)-\delta\\
&\leq C\beta^{1/2}+\delta\lt(\frac{3}{2}C\eta-1\rt),
\end{aligned}
\ee
where we used Proposition \ref{eapr.1}.

Choose $\eta$ and $\beta$ so that $\frac{3}{2}C\eta<\frac{1}{2}$ and $C\beta^{1/2}\leq\frac{\delta}{4}$
and $\delta/\beta$ large as required in the preceding paragraph. Then we obtain
\[h\leq-\frac{\delta}{4}\,\,\mbox{on $x_n=\beta$}.\]
Therefore, we can choose $A$ large such that $Ah\leq-C$ on $\{x_n=\beta\}.$
By the maximum principle we have $h\leq\pm T_\alpha u$ in $\Omega_\beta\times[0, t_0),$
so we have
\be\label{ea.42}
|u_{\alpha n}(0, t)|\leq C,\,\,\mbox{for $t\in[0, t_0).$}
\ee
Since $\partial\Omega$ is convex, $f$ satisfies \eqref{in.8}, moreover by Lemma \ref{ealm.2}, $f$ is uniformly bounded.
We can get an uniform bound on $|u_{nn}(0, t)|, t\in[0, t_0)$ by a standard contradiction argument (see \cite{CNS5}).
Thus, We complete the proof of this lemma.
\end{proof}

Combining the results in Lemma \ref{ealm.1}, \ref{ealm.3}, and \ref{ealm.4}, we proved the long time existence of
equation \eqref{ea.2}, thus Theorem \ref{eath.1}.

\bigskip
\section{Existence of the entire flow}
\label{ex}
\setcounter{equation}{0}
First of all, we will use the H{\"o}lder estimates to prove the maximality of our limit solution.
Since $f$ satisfies \eqref{in.6} and \eqref{in.7}, following \cite{SS} Section 6 we can prove
\begin{lemma}
\label{exlm.1}
Let $u:\R^{n+1}\times[0, \infty)\goto\R$ be a graphical solution to the general curvature flow and
$M\geq 1$ such that
\[|Du(x, t)|\leq M\]
for all $(x, t)$ where $u(x, t)\leq 0.$
For any $x_0\in\R^{n+1}$ and $t_1, t_2\geq 0,$ if $u(x_0, t_1)\leq -1,$ then when
$|t_1-t_2|\leq\frac{1}{8M^2}$ we have
\be\label{ex.1}
\frac{|u(x_0, t_1)-u(x_0, t_2)|}{\sqrt{|t_1-t_2|}}\leq\sqrt{2}(M+1).
\ee
\end{lemma}

Next, we are going to prove the existence result. The proof mostly follows \cite{SS} Section 8, but for completeness,
we will include it here.
\begin{theorem}
\label{exth.1}
Let $A\in\R^{n+1}$ be an open set. Assume that $u_0:A\goto\R$ is maximal, locally Lipschitz continuous,
and weakly admissible. Then there exists $\Omega\subset\R^{n+1}\times[0, \infty)$ such that
$\Omega\cap(\R^{n+1}\times\{0\})=A\times\{0\}$ and a classical singularity resolving solution
$u:\Omega\goto \R$ with initial value $u_0.$
\end{theorem}
\begin{proof}
Consider the approximate solutions $u^L_{i, R}$ to equation \eqref{ea.1}, by standard argument we get
$u^{L}_{i, i}\goto u^L$ as $i\goto\infty$ and $u^L$ is a solution to the general curvature flow with
initial condition $\min{u_0, L}.$

Let's derive the lower bound for $u^L$ that will ensure maximality of the limit when $L\goto\infty.$
As the initial value $u_0$ fulfills the maximality condition, for every $r>0$ we can find
$d=d(r)$ such that $B_r((x, L-r-1))$ lies below $\min\{u_0, L\}$ for $|x|\geq d.$ By the maximum principle we have
$u^L(x, t)\geq L-2$ for $0\leq t\leq r-1/2$ if $|x|\geq d.$ Therefore, for any $T>0$ there exists $d\geq 0$
such that $u^L(x, t)\geq L-2$ for $|x|>d$ and $0\leq t\leq T.$

By interior estimates we obtained in Section \ref{ie} and standard PDE theorem, we get locally uniform estimates on
arbitrary derivatives of $u^L$ in compact subsets of $\Omega\cap(\R^{n+1}\times(0, \infty)).$ The H{\"o}lder estimates in Lemma \ref{exlm.1}
also survives the limiting process and we obtain uniform bounds for $|u^L|_{C^{0, 1; 0, 1/2}}$ in compact subset of $\Omega.$

Now, we apply Lemma 7.3 in \cite{SS} to $u^L, L\in\mathbb{N},$ and obtain a solution $(\Omega, u)$ and a subsequence of $u^L,$ which we assume
to be $u^L$ itself, such that $u^L\goto u$ locally uniformly in $\Omega.$ By Lemma 7.1 in \cite{SS} and those interior derivative estimates we obtained
in Section \ref{ie}, we conclude that $u$ is a classical singularity resolving solution with initial value $u_0.$
\end{proof}

The relation between level set solutions (see \cite{LX2} ) and the singularity resolving solution can be derived exactly as Section 9 in \cite{SS}.
Therefore, we will only state the theorem here.
\begin{theorem}
\label{exth.2}
Let $(\Omega, u)$ be a solution to general curvature flow as in Theorem \ref{exth.1}.
Let $\partial\mathcal{D}_t$ be the level set solution of $\partial\Omega_0.$ If $\partial\mathcal{D}_t$
does not fatten, then the measure theoretic boundaries of $\Omega_t$ and $\mathcal{D}_t$ coincide for
$t>0: \partial^\mu\Omega_t=\partial^\mu\mathcal{D}_t.$
\end{theorem}

\end{document}

%% file: Figure1.pdf_tex
\begingroup%
  \makeatletter%
  \providecommand\color[2][]{%
    \errmessage{(Inkscape) Color is used for the text in Inkscape, but the package 'color.sty' is not loaded}%
    \renewcommand\color[2][]{}%
  }%
  \providecommand\transparent[1]{%
    \errmessage{(Inkscape) Transparency is used (non-zero) for the text in Inkscape, but the package 'transparent.sty' is not loaded}%
    \renewcommand\transparent[1]{}%
  }%
  \providecommand\rotatebox[2]{#2}%
  \ifx\svgwidth\undefined%
    \setlength{\unitlength}{1271.63044384bp}%
    \ifx\svgscale\undefined%
      \relax%
    \else%
      \setlength{\unitlength}{\unitlength * \real{\svgscale}}%
    \fi%
  \else%
    \setlength{\unitlength}{\svgwidth}%
  \fi%
  \global\let\svgwidth\undefined%
  \global\let\svgscale\undefined%
  \makeatother%
  \begin{picture}(1,0.69268432)%
    \put(0,0){\includegraphics[width=\unitlength]{Figure1.pdf}}%
    \put(0.76863723,0.17886506){\color[rgb]{0,0,0}\makebox(0,0)[lt]{\begin{minipage}{0.31691025\unitlength}\raggedright \(\partial\Omega_0\)\end{minipage}}}%
    \put(0.60593413,0.17887297){\color[rgb]{0,0,0}\makebox(0,0)[lt]{\begin{minipage}{0.35164013\unitlength}\raggedright \(\partial\Omega_t\)\end{minipage}}}%
  \end{picture}%
\endgroup%

%% file: Figure2.pdf_tex
\begingroup%
  \makeatletter%
  \providecommand\color[2][]{%
    \errmessage{(Inkscape) Color is used for the text in Inkscape, but the package 'color.sty' is not loaded}%
    \renewcommand\color[2][]{}%
  }%
  \providecommand\transparent[1]{%
    \errmessage{(Inkscape) Transparency is used (non-zero) for the text in Inkscape, but the package 'transparent.sty' is not loaded}%
    \renewcommand\transparent[1]{}%
  }%
  \providecommand\rotatebox[2]{#2}%
  \ifx\svgwidth\undefined%
    \setlength{\unitlength}{1394.71655452bp}%
    \ifx\svgscale\undefined%
      \relax%
    \else%
      \setlength{\unitlength}{\unitlength * \real{\svgscale}}%
    \fi%
  \else%
    \setlength{\unitlength}{\svgwidth}%
  \fi%
  \global\let\svgwidth\undefined%
  \global\let\svgscale\undefined%
  \makeatother%
  \begin{picture}(1,0.55045674)%
    \put(0,0){\includegraphics[width=\unitlength]{Figure2.pdf}}%
    \put(0.79668871,0.18653041){\color[rgb]{0,0,0}\makebox(0,0)[lt]{\begin{minipage}{0.25693414\unitlength}\raggedright \(\partial\Omega_0\)\end{minipage}}}%
    \put(0.57442154,0.15922077){\color[rgb]{0,0,0}\makebox(0,0)[lt]{\begin{minipage}{0.20808734\unitlength}\raggedright \(\partial\Omega_t\)\end{minipage}}}%
    \put(0.22760009,0.1086514){\color[rgb]{0,0,0}\makebox(0,0)[lt]{\begin{minipage}{0.20808734\unitlength}\raggedright \(\partial\Omega_t\)\end{minipage}}}%
  \end{picture}%
\endgroup%